\documentclass[11pt]{article}
\usepackage{amsmath,amsthm,amsfonts,amscd,amssymb,eucal,latexsym,mathrsfs, appendix, yhmath, fancyhdr, setspace,url}
\usepackage[numbers,sort&compress]{natbib}
\usepackage[all,cmtip]{xy}
\usepackage{abstract}
\newtheorem{theorem}{Theorem}[section]
\newtheorem{corollary}[theorem]{Corollary}
\newtheorem{lemma}[theorem]{Lemma}
\newtheorem{proposition}[theorem]{Proposition}

\theoremstyle{definition}
\newtheorem{definition}[theorem]{Definition}

\allowdisplaybreaks

\begin{document}

\title{Remarks on strong embeddability\\for discrete metric spaces and groups\footnote{The authors are supported by NSFC (No. 11231002).}}

\author{Guoqiang Li and  Xianjin Wang}

\date{}

\maketitle

\begin{abstract}
In this paper, we show that the strong embeddability has fibering permanence property and is preserved under the direct limit for the metric space. Moreover, we show the following result: let $G$ is a finitely generated group with a coarse quasi-action on a metric space $X$. If $X$ has finite asymptotic dimension and the quasi-stabilizers are strongly embeddable, then $G$ is also strongly embeddable.

\end{abstract}

\maketitle \numberwithin{equation}{section}
\section{Introduction}
In \cite{Gromov21}, Gromov introduced the notion of coarse embeddability of metric spaces and suggested that a discrete finitely generated group that coarsely embeds into a Hilbert space, when equipped with a word length metric, would satisfy the Novikov conjecture \cite{Ferry21,Gromov21}. Subsequently, in \cite{Yu21}, Yu proved the coarse Baum-Connes conjecture holds for bounded geometry discrete metric spaces which are coarsely embeddable in Hilbert space.  In the same paper, Yu introduced a weak form of amenability that he called property A, which ensures the existence of a coarse embedding into Hilbert space. In later years, property A and coarse embeddability have been further studied \cite{Bell21,Guentner21,Higson21,Nowak21,Tu21}. In \cite{Guentner21}, the author showed that property A is preserved under group extensions. Unlike property A, coarse embeddability is not closed under group extensions \cite{Arzhantseva21}. In \cite{Ji21}, Ji, Ogle and W. Ramsey introduced the notion of strong embeddability which is stable under arbitary extensions.

Strong embeddability is an intermediate notion of strong coarse embeddability implied by Property A and implying coarse embeddability. In \cite{Xia22}, J. Xia and X. Wang studied the permanence properties of strong embeddability. Moreover, they proved that a metric space is strongly embeddable if and only if it has weak finite decomposition complexity with respect to strong embeddability. In \cite{Xia21}, J. Xia and X. Wang showed that a finitely generated group acting on a finitely asymptotic dimension metric space by isometries whose {\it$k$-stabilizers} are strongly embeddable is strongly embeddable. Now we extend this result to conclude as follows.

\begin{theorem}
Assume that $G$ is a finitely generated group with a coarse quasi-action on a metric space $X$. If $X$ has finite asymptotic dimension and there exists a base point such that its all quasi-stabilizers are strongly embeddable, then $G$ is strongly embeddable.
\end{theorem}

The {\it coarse quasi-action} (see Definition \ref{25}) is designed to describe situations where elements of a group act on a metric space via {\it coarse equivalence}. And it is a significant and useful generalization of actions by isometries and quasi-isometries.

\section{Preliminaries}
A discrete metric space $X$ has bounded geometry if for all $R>0$ there exists $N_R$ such that $|B(x,R)|\leq N_R$ for all $x\in X$, where $|\cdot|$ denotes the number of elements of the ball $B(x,R)$. $X$ is called uniformly discrete if there exists a constant $C>0$ such that for any two distinct points $x,y\in X$ we have $d(x,y)\geq C$.
Assume that all metric spaces in this paper are uniformly discrete with bounded geometry. This class includes many interesting examples, in particular, all countable and uniformly discrete groups. Let $B$ be a Banach space and $B_1 =\{\eta \in B|~\|\eta\|=1\}$. For any $R,\,\epsilon >0$, a map $x\mapsto \xi_x$ from $X$ to $B$ will be said to have $(R,\,\epsilon)$ variation if $d(x,y)\leq R$ implies $\|\xi_x - \xi_y\|\leq\epsilon$.

\begin{definition}[see \cite{Ji21}]
Let $X$ be a metric space. Then $X$ is strongly embeddable if and only if for every $R,\,\epsilon>0$ there exists a Hilbert space valued map $\beta :X\rightarrow(l^2 (X))_1$ satisfying:

\vspace{3mm}(1) $\beta$ has $(R,\,\epsilon)$ variation;

\vspace{3mm}(2) $\lim\limits_{S\rightarrow \infty}\sup\limits_{x\in X}\sum\limits_{w\notin B(x,S)}\mid\beta_x (w)\mid^2=0$.~\label{def21}
\end{definition}

We will need to make use of a family of strongly embeddable metric spaces, with some uniform control.

\begin{definition}[see \cite{Ji21}]
A family $(X_i)_{i\in I}$ of metric spaces is equi-strongly embeddable if for every $R,~\epsilon>0$ there exists a family of Hilbert space valued maps $\xi^i:X_i \rightarrow(l^2 (X_i))_1$ satisfying:

\vspace{3mm}(1) for each $i\in I$, $\xi^i$ has $(R,\,\epsilon)$ variation;

\vspace{3mm}(2) $\lim\limits_{S\rightarrow \infty}\sup\limits_{i\in I}\sup\limits_{x\in X_i}\sum\limits_{w\notin B(x,S)}\mid\xi_{x}^i (w)\mid^2=0$.~\label{def22}
\end{definition}

Let $X$ be a set. A partition of unity on $X$ is a collection of functions $\{\phi_i\}_{i\in I}$, with $\phi_i :X\rightarrow [0,1]$, and such that $\sum_{i\in I}\phi_i (x)=1$ for every $x\in X$. A partition of unity $\{\phi_i\}_{i\in I}$ is said to subordinate to a cover $\mathcal{U}=\{U_i\}_{i\in I}$ of $X$ if each $\phi_i$ vanishes outside $U_i$.

The following conclusion will be useful to prove the main theorem.
\begin{theorem}[see \cite{Xia21}]
Say $X$ is a metric space such that for any $R,\,\epsilon>0$ there exists a partition of unity $\{\phi_i\}_{i\in I}$ on $X$ satisfying:

\vspace{3mm}(1) for all $x,y\in X$, if $d(x,y)\leq R$, then $\sum\limits_{i\in I}|\phi_{i}(x)-\phi_{i}(y)|\leq\epsilon$;

\vspace{3mm}(2) $\{\phi_i\}_{i\in I}$ is subordinated to an equi-strongly embeddable cover $\mathcal{U}=\{U_i\}_{i\in I}$ of $X$.

\vspace{2mm}Then $X$ is strongly embeddable.~\label{thm21}
\end{theorem}

Let $\mathcal{U}=\{U_i\}_{i\in I}$ be a cover of a metric space $X$. We say the multiplicity of $\mathcal{U}$ is $k$ if each point of $X$ is contained in at most $k$ elements of $\mathcal{U}$. The $R$-multiplicity of $\mathcal{U}$ is the maximum number of elements of $\mathcal{U}$ that meet a commom ball of radius $R$ in $X$ $(R>0)$. The Lebesgue number of $\mathcal{U}$ is $L$ if any ball of radius at most $L$ is contained in one element of $\mathcal{U}$. If $d(U,V)>L$ for all $U,V\in \mathcal{U}$ with $U\neq V$ then $\mathcal{U}$ is $L$-separated $(L>0)$. A cover $\mathcal{U}$ of $X$ is $(k,L)$-separated $(k\geq 0 ~\rm{and}~ L>0)$ if there is a partition of $\mathcal{U}$ into $k+1$ families $$\mathcal{U}=\mathcal{U}_0 \cup\cdots\cup\mathcal{U}_k$$ such that each family $\mathcal{U}_i$ is $L$-separated. Note that a $(k,2L)$-separated cover has $L$-multiplicity $\leq k+1$. Let
$$\mathcal{U}_L =\{U(L)|U\in\mathcal{U}\},\,U(L)=\{x\in X|d(x,U)\leq L\}.$$
Note then that: if a cover $\mathcal{U}$ of $X$ has $L$-multiplicity $\leq k+1$ then the enlarged cover $\mathcal{U}_L$ has multiplicity $\leq k+1$ and lebesgue number $L$.

\begin{definition}[see \cite{Rufus21}]
Let $X$ be a metric space. $X$ is said to have finite asymptotic dimension if there exists $k\geq 0$ such that for all $L\geq 0$ there exists a uniformly bounded cover of $X$ of Lebesgue number at least $L$ and multiplicity $k+1$. The least possible such $k$ is the asymptotic dimension of $X$.~\label{def23}
\end{definition}
The following result is part of the folklore of the subject.
\begin{lemma}[see \cite{Bell22}]
Let $X$ be a metric space and $\mathcal{U}=\{U_i\}_{i\in I}$ be a cover of $X$ with multiplicity $k$ and lebesgue number $L$. Then there is a partition of unity  $\{\phi_i\}_{i\in I}$ subordinating to the cover such that $$\sum\limits_{i\in I}|\phi_i (x)-\phi_i (y)|\leq\frac{(2k+2)(2k+3)}{L}d(x,y)$$ for any $x,y\in X$.~\label{lem21}
\end{lemma}
Next we will recall the notion of {\it coarse quasi-action} which is a generalization of the notion of {\it quasi-action} central to the fundamental problem of quasi-isometry classification of finitely generated groups, see \cite{Drutu21,Kapovich21,Mosher21}.
\begin{definition}[see \cite{Nowak21}]
Let $X$ and $Y$ be metric space and $f:X\rightarrow Y$ be a map.

\vspace{2mm}(1) The map $f$ is called proper if the inverse image, under $f$, of each bounded subset of $Y$, is a bounded subset of $X$.

\vspace{2mm}(2) The map $f$ is called bornologous if there exists a real positive non-decreasing function $\ell:[0,+\infty)\rightarrow[0,+\infty)$ such that $$d(x,x')\leq r\Rightarrow d(f(x),f(x'))\leq\ell(r).$$

\vspace{2mm}(3) The map $f$ is called coarse if it is both proper and bornologous.~\label{def24}

\end{definition}
We say that two maps $f,f':X\rightarrow Y$ are close if $d(f,f')=\sup\limits_{x\in X}d(f(x),f'(x))$ is bounded. $X$ and $Y$ are coarsely equivalent if there exist coarse maps $f:X\rightarrow Y$ and $g:Y\rightarrow X$ such that $f\circ g$ and $g\circ f$ are close to the identity maps on $Y$ and on $X$ respectively.
\begin{definition}[see \cite{Beck21}]
A coarse quasi-action of a group $G$ on a metric space $X$ is an assignment of a coarse self-equivalence $f_g :X\rightarrow X$ for each element $g\in G$ such that the following conditions are satisfied:

\vspace{2mm}(1) all $f_g$ are coarse maps with respect to a uniform choice of the function $\ell$;

\vspace{2mm}(2) there exists a number $A\geq 0$ such that $d(f_{id},id_X)\leq A$;

\vspace{2mm}(3) there exists a number $B\geq 0$ such that $d(f_{g}\circ f_h ,f_{gh})\leq B$ for all elements $g,h\in G$.~\label{25}

\end{definition}

{\noindent}It is immediate from the above that $d(f_{g}\circ f_{-g} ,id_X)\leq{A+B}$ for all $g\in G$.

\section{Direct limits for the metric space}
In this section, we will prove that the strong embeddability is closed under the direct limit for the metric space. Note that this result is rather different form (\cite{Xia22}, Theorem 4.3). The group structure on the direct limit is essential to enable one to drop the boundedness assumptions necessary for this result.

First, we have the following fact.
\begin{lemma}
Let $\{X_i\}_{i\in I}$ be a collection of equi-strongly embeddable metric spaces. If $Y_i$ is the subspace of $X_i$ for each $i$, then $\{Y_i\}_{i\in I}$ is also equi-strongly embeddable.~\label{Lemma22}
\end{lemma}
\begin{proof}
For each $i\in I$, suppose $d_i$ is the metric on $X_i$. Let $p_i :X_i \rightarrow Y_i$ be a function defined by
$$p_i (x)= \begin{cases} y, &\rm{if}~~x\in X_i\setminus Y_i,~d_i (x,y)\leq 2\,d_i (x,Y_i);\\
   x, &\rm{if}~~x\in Y_i. \end{cases}$$
Let $R>0$ and $\epsilon >0$, as in Definition \ref{def22}, there exists a family of maps $\beta^i:X_i \rightarrow(l^2 (X_i))_1$ satisfying:

\vspace{2mm}(1) for each $i\in I$, $\beta^i$ has $(R,~\epsilon)$ variation;

\vspace{2mm}(2) $\lim\limits_{S\rightarrow \infty}\sup\limits_{i\in I}\sup\limits_{x\in X_i}\sum\limits_{w\notin B(x,S)}\mid\beta_{x}^i (w)\mid^2=0$.

Then, for each $i\in I$ we define an isometry $\alpha^i :\ell^2 (X_i)\rightarrow \ell^2 (Y_i \times X_i)$ by

$$\alpha^i (\zeta)(y,x)= \begin{cases} \zeta(x), &\rm{if}~~y=p_i (x);\\
   0, &\rm{if}~~otherwise. \end{cases}$$
for each $\zeta\in\ell^2 (X_i)$.

Define $\xi^i :Y_i \rightarrow \ell^2 (Y_i \times X_i)$ by $\xi_y ^i (t,s)=\alpha^i (\beta_y ^i)(t,s)$ for $t\in Y_i,\,s\in X_i$.
Note that for any $y\in Y_i$,
\begin{align*}
\|\xi_y ^i\|^2_{\ell^2 (Y_i \times X_i)}=&\sum\limits_{(t,s)\in Y_i \times X_i}|\xi_y ^i (t,s)|^2\\
=&\sum\limits_{(t,s)\in Y_i \times X_i}|\alpha^i (\beta_y ^i)(t,s)|^2\\
=&\sum\limits_{s\in X_i}|\beta_y ^i(s)|^2\\
=&\|\beta_y ^i\|_{\ell^2 (X_i)}^2\\
=&1.
\end{align*}
Similarly, for any $y,y'\in Y_i$,
$$\|\xi_y ^i -\xi_{y'}^i\|_{\ell^{2}(Y_i \times X_i)}^2 =\|\beta_y ^i -\beta_{y'}^i\|_{\ell^{2}(X_i)}^2$$
Now we define $\eta^i :Y_i \rightarrow \ell^2 (Y_i)$ by
$$\eta_{y}^{i}(t)=\|\xi_{y}^{i}(t,\cdot)\|_{\ell^2 (X_i)}$$
for $t\in Y_i$.

Note that for any $y\in Y_i$,
$$\|\eta_{y}^i\|^2 _{\ell^2 (Y_i)}=\sum\limits_{t\in Y_i}|\eta_{y}^{i}(t)|^2=\sum\limits_{t\in Y_i}\|\xi_{y}^{i}(t,\cdot)\|^2 _{\ell^2 (X_i)}=\|\xi_y ^i\|^2 _{\ell^{2}(Y_i \times X_i)} =1.$$
Then for any $y,y'\in Y_i$ and $d_i (y,y')\leq R$, we have
\begin{align*}
\|\eta_y ^i -\eta_{y'}^i\|_{\ell^{2}(Y_i)}^2=&\sum\limits_{t\in Y_i}|\eta_{y}^i (t)-\eta_{y'}^i (t)|^2\\
=&\sum\limits_{t\in Y_i}|~\|\xi_{y}^{i}(t,\cdot)\|_{\ell^2 (X_i)}-\|\xi_{y'}^{i}(t,\cdot)\|_{\ell^2 (X_i)}|^2\\
\leq&\sum\limits_{t\in Y_i}\|\xi_{y}^{i}(t,\cdot)-\xi_{y'}^{i}(t,\cdot)\|^2 _{\ell^2 (X_i)}\\
=&\|\xi_y ^i -\xi_{y'}^i\|_{\ell^{2}(Y_i \times X_i)}^2\\
=&\|\beta_y ^i -\beta_{y'}^i\|_{\ell^{2}(X_i)}^2\\
\leq&\epsilon.
\end{align*}
It follows that $\eta^i$ has $(R,\,\epsilon)$ variation.

{\noindent}Moreover,
\begin{align*}
&\lim\limits_{S\rightarrow \infty}\sup\limits_{i\in I}\sup\limits_{y\in Y_i}\sum\limits_{z\notin B(y,S)}\mid\eta_{y}^i (z)\mid^2\\
=&\lim\limits_{S\rightarrow \infty}\sup\limits_{i\in I}\sup\limits_{y\in Y_i}\sum\limits_{z\notin B(y,S)}\|\xi_{y}^{i}(z,\cdot)\|^{2}_{\ell^{2}(X_i)}\\
=&\lim\limits_{S\rightarrow \infty}\sup\limits_{i\in I}\sup\limits_{y\in Y_i}\sum\limits_{z\notin B(y,S)}\sum\limits_{x\in X_i}|\alpha^{i}(\beta_{y}^i)(z,x)|^2\\
=&\lim\limits_{S\rightarrow \infty}\sup\limits_{i\in I}\sup\limits_{y\in Y_i}\sum\limits_{x\in X_i}\sum\limits_{z\notin B(y,S)}|\alpha^{i}(\beta_{y}^i)(z,x)|^2\\
=&\lim\limits_{S\rightarrow \infty}\sup\limits_{i\in I}\sup\limits_{y\in Y_i}\sum\limits_{x\in Y_i}\sum\limits_{z\notin B(y,S)}|\alpha^{i}(\beta_{y}^i)(z,x)|^2\\&+\lim\limits_{S\rightarrow \infty}\sup\limits_{i\in I}\sup\limits_{y\in Y_i}\sum\limits_{x\in X_i\setminus Y_{i}}\sum\limits_{z\notin B(y,S)}|\alpha^{i}(\beta_{y}^i)(z,x)|^2\\
=&\lim\limits_{S\rightarrow \infty}\sup\limits_{i\in I}\sup\limits_{y\in Y_i}\sum\limits_{x\notin B(y,S)}|\beta_{y}^i (x)|^2 +\lim\limits_{S\rightarrow \infty}\sup\limits_{i\in I}\sup\limits_{y\in Y_i}\sum\limits_{x\in X_i \setminus Y_i}|\beta_{y}^i (x)|^2\\
\leq&\lim\limits_{S\rightarrow \infty}\sup\limits_{i\in I}\sup\limits_{y\in X_i}\sum\limits_{x\notin B(y,S)}|\beta_{y}^i (x)|^2 +\lim\limits_{S\rightarrow \infty}\sup\limits_{i\in I}\sup\limits_{y\in X_i}\sum\limits_{x\notin B(y,S)}|\beta_{y}^i (x)|^2\\
=&0.
\end{align*}
Then the claim follows.

\end{proof}

\begin{lemma}
Let $c>0$. Let $\{X_i\}_{i\in I}$ be a collection of metric spaces and $Y_i$ is a $c$-net of $X_i$ for each $i\in I$. If $\{Y_i\}_{i\in I}$ is equi-strongly embeddable, then so is $\{X_i\}_{i\in I}$.~\label{Lemma23}
\end{lemma}
\begin{proof}
This proof is analogous to (\cite{Xia22}, Lemma 5.2), so we omit it.
\end{proof}

\begin{proposition}
Let $X_1\subseteq X_2\subseteq X_3\subseteq\cdots$ be an increasing sequence of bounded metric space, and let $X=\bigcup_{n=1}^\infty X_n$. Assume also that any bounded subset of $X$ is contained in some $X_n$. If $\{X_n\}_{n=1}^\infty$ is equi-strongly embeddable, then $X$ is strongly embeddable.

\end{proposition}
\begin{proof}
Take $L>0$. For given $X_{n_k}$, since each $X_n$ is bounded, we can choose \vspace{2mm}$X_{n_{k+1}}$ satisfying $\bigcup\limits_{x\in X_{n_k}}\bar{B}(x,3L)\subseteq X_{n_{k+1}}$.
Thus we obtain a subsequence $\{X_{n_k}\}$ of $\{X_n\}$ such that $$(\bigcup\limits_{x\in X_{n_k}}\bar{B}(x,L))\,\bigcap \,(\bigcup\limits_{x\in X_{n_{k+2}}\backslash X_{n_{k+1}}}\bar{B}(x,L))=\varnothing.$$
Set $$U_k=\bigcup\limits_{x\in X_{n_{k+1}}\setminus X_{n_k}}\bar{B}(x,L).$$
for every $k\geq 1$.

Then we obtain a cover $\mathcal{U}=\{U_k\}_{k=1}^\infty$ of multiplicity at most 2 and Lebesgue number at least $L$. We claim that $\mathcal{U}$ is equi-strongly embeddable. Indeed, $\{X_{n_{k+1}}\backslash X_{n_k}\}$ are subspaces of the equi-strongly embeddable sequence $\{X_{n_k}\}$. It follows from Lemma \ref{Lemma22} that $\{X_{n_{k+1}}\backslash X_{n_k}\}$ is equi-strongly embeddable. Note that for every $k$, $\{X_{n_{k+1}}\backslash X_{n_k}\}$ is a $L$-net of $U_k$. By lemma \ref{Lemma23}, $\{U_k\}_{k=1}^\infty$ is equi-strongly embeddable.

By choosing $L$ large enough, the result follows from Lemma \ref{lem21} and Theorem \ref{thm21}.
\end{proof}

\section{Fibering permanence}
In this section, we show an important property of strong embeddability which is called fibrering permanence. The main motivating examples are fibre bundles $p:X\rightarrow Y$ with base space $Y$ and total space $X$. Moreover, fibering permanence is somewhat more subtle than the other permanence properties, and care must be taken to formulate it correctly for our other basic properties.
\begin{proposition}
Let $X$ and $Y$ be two metric spaces and $f:X\rightarrow Y$ be an uniformly expansive map. Assume that $Y$ has Property A. If for every uniformly bounded cover $\{U_i\}_{i\in I}$ of $Y$, the collection $\{f^{-1}(U_i)\}_{i\in I}$ of subspaces of $X$ is equi-strongly embeddable, then $X$ is strongly embeddable.~\label{prop21}

\end{proposition}

\begin{proof}
Let $R,\epsilon >0$ be given. Since $f$ is uniformly expansive, there exists $S>0$, such that $d(f(x),f(y))\leq S$ whenever $d(x,y)\leq R$. Since $Y$ has property A, by one of the equivalent definitions of property A (see \cite{Rufus21}), there exists an uniformly bounded cover $\mathcal{U}=\{U_i\}_{i\in I}$ of $Y$, together with a partition of unity $\{\phi_i\}_{i\in I}$ subordinated to $\mathcal{U}$, such that $$\sum\limits_{i\in I}|\phi_{i}(y)-\phi_{i}(y')|\leq\epsilon$$ for $y,y'\in Y$, $d(y,y')\leq S$.

We define $\varphi_i =\phi_i \circ f$ for each $i\in I$. Clearly, $\{\varphi_i\}_{i\in I}$ is a partition of unity on $X$ subordinated to $\{f_{-1}(U_i)\}_{i\in I}$ and satisfying the assumptions of Theorem \ref{thm21}. Thus $X$ is strongly embeddable.
\end{proof}
\begin{corollary}
Let $X$ and $Y$ be two metric spaces and $f:X\rightarrow Y$ be a Lipschitz map of metric spaces. Suppose that a group $G$ acts by isometries on both $X$ and $Y$, that the action on $Y$ is transtitive and that $f$ is $G$-equivariant. Suppose that $Y$ has property A. If there exists $y_0 \in Y$ satisfying for every $n\in\mathbb{N}$ the inverse image $f^{-1}(B(y_0 ,n))$ is strongly embeddable, then $X$ is strongly embeddable.
\end{corollary}
\begin{proof}
Let $\{U_i\}_{i\in I}$ be a uniformly bounded cover of $Y$. Note that the action of $G$ on $Y$ is isometrical and transitive, there exists $n\in \mathbb{N}$ and $g_i \in G$ such that $g_i U_i\subseteq B(y_0 ,n)$ for all $i\in I$. Note also that $g_i f^{-1}(U_i)=f^{-1}(g_i U_i)\subseteq f^{-1}(B(y_0 ,n))$, we have that the collection $\{f^{-1}(U_i)\}_{i\in I}$ is isometric to a collection of subspaces of $f^{-1}(B(y_0 ,n))$.

Continuing, note that $f^{-1}(B(y_0 ,n))$ is strongly embeddable, then we have that\\
{\noindent}$\{f^{-1}(U_i)\}_{i\in I}$ is equi-strongly embeddable. Now the result immediately follows from Theorem \ref{prop21}.
\end{proof}
A metric space of finite asymptotic dimension has property A, also is strongly embeddable. Now we prove a natural generalization of this result, where uniform boundedness of the cover is replaced by the appropriate uniform version of strong embeddability.
\begin{theorem}
Let $X$ be a metric space. If for any $\sigma>0$ there exists a $(k,L)$-separated cover $\mathcal{U}$ of $X$ with $k^2 +1\leq L\sigma$ and $\mathcal{U}$ is equi-strongly embeddable, then $X$ is strongly embedable.
\end{theorem}
\begin{proof}
Let $R,\epsilon>0$. Take a number $\sigma$, such that $0<\sigma<1/20R$. Then for any integer $k\geq 0$, we have $$k^2 +1\geq 2(2k+2)(2k+3)R\sigma.$$
It follows from the assumption that there exists a $(k,2L)$-separated cover $\mathcal{U}$ of $X$ such that $\mathcal{U}$ is equi-strongly embeddable and $k^2 +1\leq 2L\sigma\epsilon$. Note that the cover $\mathcal{U}_L$ has multiplicity $\leq k+1$ and Lebesgue number $L$. Moreover, $\mathcal{U}_L$ is equi-strongly embeddable, as it is coarsely equivalent to $\mathcal{U}_L$. By Lemma \ref{lem21}, there is a partition of unity $\{\phi_{\mathcal{U}(L)}\}_{\mathcal{U}(L)\in \mathcal{U}_L}$ subordinated to $\mathcal{U}_L$ such that for all $x,y\in X$
\begin{align*}
\sum\limits_{\mathcal{U}(L)\in \mathcal{U}_L}|\phi_{\mathcal{U}(L)}(x)-\phi_{\mathcal{U}(L)}(y)|\leq&\frac{(2k+2)(2k+3)}{L} d(x,y)\\
\leq&\frac{k^2 +1}{2RL\sigma}d(x,y)
\end{align*}
In particular, if $d(x,y)\leq R$, we have that $$\sum\limits_{\mathcal{U}(L)\in \mathcal{U}_L}|\phi_{\mathcal{U}(L)}(x)-\phi_{\mathcal{U}(L)}(y)|
\leq\frac{k^2 +1}{2L\sigma}\leq\epsilon.$$
This shows that $X$ satisfies the conditions in Theorem \ref{thm21}. Hence, $X$ is strongly embeddable.
\end{proof}

\section{Groups acting on metric spaces}
In this section, we are ready to complete the proof of our main result.

First, we recall the notion of $T$-quasi-stabilizer $(T>0)$. Let $G$ be a finitely generated group with a coarse quasi-action on a metric space $X$ and $x_0$ be a chosen base point in $X$. The $T$-quasi-stabilizer $W_{T}(x_0)$ is defined to be the subset of all elements $g$ in $G$ such that $d(gx_0 ,x_0)\leq T$. Moreover, we can view $G$ as a metric space with a word length metric (see \cite{Nowak21}).
\begin{theorem}
Assume that $G$ is a finitely generated group with a coarse quasi-action on a metric space $X$. If $X$ has finite asymptotic dimension and there exists a base point $x_0 \in X$ such that $W_T (x_0)$ is strongly embeddable for any $T>0$. Then $G$ is strongly embeddable.~\label{thm22}
\end{theorem}
\begin{proof}
Since the orbit $Gx_0$ is a subset of $X$, $Gx_0$ has finite asymptotic dimension. Without loss of generality, we can assume that the action of $G$ on $X$ is transitive. Let $S$ be the finite symmetric generating set in the definition of the word length metric for $G$ and take
$$\lambda =\max\{d(sx_0 ,x_0)\mid s\in S\}.$$
Then there exists a map $\pi :G\rightarrow X$ given by $\pi (g)=gx_0$ for all $g\in G$. If the action of $G$ on $X$ is by isometries, $\pi$ is $\lambda$-Lipschitz. In our case, we show that $\pi$ is $\ell(\lambda)$-Lipschitz. Indeed, $d(\pi(g),\pi(gs))=d(gx_0 ,gsx_0)\leq\ell(d(x_0 ,sx_0))\leq\ell(\lambda)$ for any $g\in G$ and $s\in S$.

Suppose $X$ has asymptotic dimension $\leq k$. Let $L>0$ be given. By definition \ref{def23}, there exists a uniformly bounded cover $\mathcal{U}=\{U_i\}_{i\in I}$ of $X$ with Lebesgue number $L$ and multiplicity $k+1$ such that the $L$-neighbourhood $\mathcal{V}=\{V_i\}_{i\in I}$
of $\mathcal{U}$ is also a cover of multiplicity $k+1$.

Since $\mathcal{V}$ is uniformly bounded, there exist $T>0$ and $x_i\in X$ such that $V_i\subseteq B(x_i ,T)$ for each $i\in I$. On the other hand, we can take $g_i\in G$ such that $x_i =g_i x_0$.

By Definition \ref{25}, we have that $d(g_{i}^{-1}x_i ,x_0)\leq A+B$, then
$$g_{i}^{-1}(B(x_i ,T))\subseteq B(g_{i}^{-1}x_i ,\ell(T))\subseteq B(x_0 ,A+B+\ell(T)).$$
It follows from the definition of $\pi$,
\begin{align*}
g_{i}^{-1}\pi^{-1}(B(x_i ,T))\subseteq&\pi^{-1}(B(x_0 ,A+2B+\ell(T)))\\
=&W_{A+2B+\ell(T)}(x_0)
\end{align*}
Then we get $$g_{i}^{-1}(V_i)\subseteq W_{A+2B+\ell(T)}(x_0).$$
Note that $\{\pi^{-1}(V_i)\}_{i\in I}$ is isometric to a family of subspace of $W_{A+2B+\ell(T)}(x_0)$. Since $W_{A+2B+\ell(T)}(x_0)$ is strongly embeddable, we have that $\{\pi^{-1}(V_i)\}_{i\in I}$ is equi-strongly embeddable. By the same argument, $\{\pi^{-1}(U_i)\}_{i\in I}$ is also equi-strongly embeddable and covers $G$.

$\vspace{2mm}$We will use Theorem \ref{thm21} to complete the proof. Let $R>0$ and $\epsilon>0$. Take $L\geq\vspace{2mm}\frac{2\ell(\lambda)R(2k+2)(2k+3)}{\epsilon}$. Since $\mathcal{U}$ is a uniformly bounded cover of $X$ with Lebesgue number $L$ and multiplicity $k+1$, it follows from Lemma \ref{lem21} that there exists a partition of unity $\{\phi_{U_i}\}_{u_i \in\mathcal{U}}$ subordinated to the cover satisfying $$\sum\limits_{U_i \in \mathcal{U}}|\phi_{U_i}(x)-\phi_{U_i}(y)|\leq\frac{(2k+2)(2k+3)}{L} d(x,y)$$ for any $x,y\in X$.

Moreover, note that $\{\pi^{-1}(V_i)\}_{i\in I}$ is equi-strongly embeddable, where there exists a collection of maps $$\beta^{i}:\pi^{-1}(V_i)\rightarrow(\ell^{2}(\pi^{-1}(V_i)))_1$$ such that $\beta^{i}$ has $(R,\frac{\epsilon}{4})$ variation for every $i\in I$.

Define for each $i\in I$ a map $\varphi_i :G\rightarrow[0,1]$ by setting $$\varphi_i (g)=\sum\limits_{h\in\pi^{-1}(V_i)}\phi_{U_i}(\pi(g))|\beta_{g}^{i}(h)|^2$$
for $g\in G$.
We claim that $\{\varphi_i\}_{i\in I}$ is partition of unity on $G$ satisfying the conditions of Theorem \ref{thm21}.
First, for any $g\in G$,
\begin{align*}
\sum\limits_{i\in I}\varphi_i (g)=&\sum\limits_{i\in I}\sum\limits_{h\in\pi^{-1}(V_i)}\phi_{U_i}(\pi(g))|\beta_{g}^{i}(h)|^2\\
=&\sum\limits_{i\in I}\phi_{U_i}(\pi(g))\sum\limits_{h\in\pi^{-1}(V_i)}|\beta_{g}^{i}(h)|^2\\
=&\sum\limits_{i\in I}\phi_{U_i}(\pi(g))\\
=&1.
\end{align*}
Note also that each $\varphi_i$ vanishes outside $\pi^{-1}(U_i)$, so $\{\varphi_i\}_{i\in I}$ is a partition of unity on $G$ and subordinates to $\{\pi^{-1}(U_i)\}_{i\in I}$. Second, for any $g,g'\in G$ with $d(g,g')\leq R$. If $g\in \pi^{-1}(U_i)\}_{i\in I}$, then $\pi(g')$ belongs to the $\ell(\lambda)R$-neighbourhood of $U_i$. The latter space is a subspace of $V_i$ as $L$ is large enough. Then $g'\in\pi^{-1}(V_i)$. We have the following estimates
\begin{align*}
&\sum\limits_{i\in I}|\varphi_{i}(g)-\varphi_{i}(g')|\\
=&\sum\limits_{i\in I}|\sum\limits_{h\in\pi^{-1}(V_i)}\phi_{U_i}(\pi(g))|\beta_{g}^{i}(h)|^2 -\sum\limits_{h\in\pi^{-1}(V_i)}\phi_{U_i}(\pi(g'))|\beta_{g'}^{i}(h)|^2~|\\
=&\sum\limits_{i\in I}|\sum\limits_{h\in\pi^{-1}(V_i)}(\phi_{U_i}(\pi(g))|\beta_{g}^{i}(h)|^2 -\phi_{U_i}(\pi(g'))|\beta_{g'}^{i}(h)|^2)~|\\
\leq&\sum\limits_{i\in I}\sum\limits_{h\in\pi^{-1}(V_i)}|~\phi_{U_i}(\pi(g))|\beta_{g}^{i}(h)|^2 -\phi_{U_i}(\pi(g'))|\beta_{g'}^{i}(h)|^2~|\\
\leq&\sum\limits_{i\in I}\sum\limits_{h\in\pi^{-1}(V_i)}|~\phi_{U_i}(\pi(g))|\beta_{g}^{i}(h)|^2 -\phi_{U_i}(\pi(g))|\beta_{g'}^{i}(h)|^2\\&+\phi_{U_i}(\pi(g))|\beta_{g'}^{i}(h)|^2 -\phi_{U_i}(\pi(g'))|\beta_{g'}^{i}(h)|^2~|\\
\leq&\sum\limits_{i\in I}\sum\limits_{h\in\pi^{-1}(V_i)}\phi_{U_i}(\pi(g))|~|\beta_{g}^{i}(h)|^2 -|\beta_{g'}^{i}(h)|^2~|\\&+\sum\limits_{i\in I}|\phi_{U_i}(\pi(g))-\phi_{U_i}(\pi(g'))|\\
\leq&\sum\limits_{i\in I}\phi_{U_i}(\pi(g))\sum\limits_{h\in\pi^{-1}(V_i)}|\beta_{g}^{i}(h)+\beta_{g'}^{i}(h)|\cdot|\beta_{g}^{i}(h)-\beta_{g'}^{i}(h)|\\
&+\frac{(2k+2)(2k+3)}{L} d(\pi(g),\pi(g'))\\
\leq&\sum\limits_{i\in I}\phi_{U_i}(\pi(g))[(\sum\limits_{h\in\pi^{-1}(V_i)}|\beta_{g}^{i}(h)+\beta_{g'}^{i}(h)|^2 )^{\frac{1}{2}}\cdot(\sum\limits_{h\in\pi^{-1}(V_i)}
|\beta_{g}^{i}(h)-\beta_{g'}^{i}(h)|^2 )^{\frac{1}{2}}]\\
&+\frac{(2k+2)(2k+3)}{L}\ell(\lambda)R\\
\leq&\sum\limits_{i\in I}\phi_{U_i}(\pi(g))\parallel\beta_{g}^{i}+\beta_{g'}^{i}\parallel\cdot\parallel\beta_{g}^{i}-\beta_{g'}^{i}\parallel+\frac{\epsilon}{2}\\
\leq&\sum\limits_{i\in I}\phi_{U_i}(\pi(g))\cdot 2\cdot\frac{\epsilon}{4}+\frac{\epsilon}{2}\\
=&\epsilon.
\end{align*}
Thus $\{\varphi_i\}_{i\in I}$ has the properties in Theorem \ref{thm21}, whence $G$ is strongly embeddable.
\end{proof}
Since Theorem \ref{thm21} also holds for coarse embeddablity and exactness (it is equivalent to property A for the metric spaces with bounded geometry) \cite{Guentner21}, a similar proof shows that the same hypotheses, with 'coarse embeddability' (resp. exactness) replacing 'strong embeddability', imply that $G$ is coarsely embeddable (resp. exact), as described next.
\begin{theorem}[see \cite{Guentner21}]
Say $X$ is a metric space such that for any $R,\,\epsilon>0$ there exists a partition of unity $\{\phi_i\}_{i\in I}$ on $X$ satisfying:

\vspace{2mm}(1) for all $x,y\in X$, if $d(x,y)\leq R$, the $\sum\limits_{i\in I}|\phi_{i}(x)-\phi_{i}(y)|\leq\epsilon$;

\vspace{1mm}(2) $\{\phi_i\}_{i\in I}$ is subordinated to an equi-coarsely embeddable (resp. equi-exact) cover

\vspace{2mm}~~~~~$\mathcal{U}=\{U_i\}_{i\in I}$ of $X$.

\vspace{2mm}Then $X$ is coarsely embeddable (resp. exact).
\end{theorem}
\begin{theorem}
Assume that $G$ is a finitely generated group with a coarse quasi-action on a metric space $X$. If $X$ has finite asymptotic dimension and there exists a base point $x_0 \in X$ such that $W_T (x_0)$ is coarsely embeddable (resp. exact) for any $T>0$. Then $G$ is coarsely embeddable (resp. exact).
\end{theorem}


\vskip 1cm

\noindent \noindent Guoqiang Li\\
College of Mathematics and Statistics,\\
Chongqing University (at Huxi Campus),\\
Chongqing 401331, P. R. China\\
E-mail: \url{guoqiangli@cqu.edu.cn}\\

\noindent \noindent Xianjin Wang\\
College of Mathematics and Statistics,\\
Chongqing University (at Huxi Campus),\\
Chongqing 401331, P. R. China\\
E-mail: \url{xianjinwang@cqu.edu.cn}\\

\begin{thebibliography}{99}

\bibitem{Arzhantseva21}
G.~Arzhantseva and R.~Tessera, {\it Admitting a coarse embedding is not preserved under group extensions}, arXiv:1605.01192.

\bibitem{Beck21}
S.~Beckhardt and B.~Goldfarb, {\it Extension properties of asymptotic property C and finite decomposition complexity}, arXiv:1607.00445.

\bibitem{Bell21}
G.~Bell, {\it Property A for groups acting on metric spaces}, Topology Appl. {\bf130}(2003), 239¨C251.


\bibitem{Bell22}
G.~Bell, {\it Asymptotic properties of groups acting on complexes}, Proc. Amer. Math. Soc. {\bf 133}(2005), no. 2, 387¨C396.

\bibitem{Guentner21}
M.~Dadarlat and E.~Guentner, {\it Uniform embeddability of relatively hyperbolic groups}, J.~Reine Angew. Math. {\bf612}(2007), 1-15.

\bibitem{Drutu21}
C.~Drutu, {\it Quasi-isometry rigidity of groups. in G$\acute{e}$om$\acute{e}$tries $\grave{a}$ courbure n$\acute{e}$gative ou nulle,
groupes discrets et rigidit$\acute{e}$s}, in S$\rm{\acute{e}}$min. Congr., 18, Soc. Math. France, Paris, 2009, 321-371.

\bibitem{Ferry21}
S.~Ferry, A.~Ranicki, and J.~Rosenberg (eds.), {\it Novikov conjectures, index theorems and rigidity}, London
Mathematical Society Lecture Notes, no. 226, 227, Cambridge University Press, 1995.

\bibitem{Gromov21}
M.~Gromov, {\it Asymptotic Invariants of Infnite Groups}, Volume 2 of ``Geometry Group Theory, Sussex 1991", G. A.~Niblo and M. A.~Roller Eds, Cambridge Univ. Press, 1993.

\bibitem{Higson21}
N.~Higson and J.~Roe, {\it Amenable group actions and the Novikov conjecture}, J. Reine Angew. Math. {\bf 519}(2000), 143¨C153.

\bibitem{Ji21}
R.~Ji, C.~Ogle and B. W.~Ramsey, {\it Strong embeddablility and extensions of groups}, arXiv:1307.1935.

\bibitem{Kapovich21}
M.~Kapovich, {\it Lectures on quasi-isometric rigidity, in Geometric Group Theory}, volume 21
of Publications of IAS/Park City Summer Institute, Amer. Math. Soc., Providence, RI, 2014, 127-172.

\bibitem{Mosher21}
L.~Mosher, M.~Sageev and K.~Whyte, {\it Quasi-actions on trees I. Bounded valence}, Annals Math. {\bf158}(2003), 115-164.

\bibitem{Nowak21}
P.~Nowak and G.~Yu, {\it Large Scale Geometry}, European Mathematical Society, 2012.

\bibitem{Tu21}
Jean-Louis Tu, {\it Remarks on Yu's property A for disrete metric spaces and groups}, Bull. Soc. Math. France. {\bf 129}(2001), no. 1, 115-139.

\bibitem{Rufus21}
R.~Willett, {\it Some notes on property A}. EPFL Press, Lansanne, 2009.

\bibitem{Xia22}
J.~Xia and X.~Wang, {\it On strong embeddability and finite decomposition complexity}, Acta Math. Sin., Engl. Ser. {\bf 33}(2017), 403-418.

\bibitem{Xia21}
J.~Xia and X.~Wang, {\it Strong embeddability for groups acting on metric spaces}, In press, 2017.

\bibitem{Yu21}
G.~Yu, {\it The coarse Baum-connes conjecture for spaces which admit a uniform embedding into Hilbert space}, Invent. Math. {\bf139}(2000), 201-240.



\end{thebibliography}
\end{document}